\documentclass[runningheads]{llncs}

\usepackage{graphicx}
\usepackage{amsmath, amssymb}
\usepackage{graphicx}
\usepackage{makecell}
\usepackage[colorlinks=true, allcolors=blue]{hyperref}
\usepackage{cleveref}
\usepackage[all]{xy}
\usepackage{booktabs}
\DeclareMathOperator{\Either}{Either}
\DeclareMathOperator{\MLR}{MLR}
\DeclareMathOperator{\KLR}{KLR}
\DeclareMathOperator{\DIM}{DIM}
\DeclareMathOperator{\aeKLR}{a.e.-KLR}
\DeclareMathOperator{\aeMLR}{a.e.-MLR}

\DeclareMathOperator{\Some}{Some}

\DeclareMathOperator{\Many}{Many}
\newcommand{\C}{\mathcal C}
\newcommand{\D}{\mathcal D}
\newcommand{\upto}{\upharpoonright}

\begin{document}
\title{Strong Medvedev reducibilities and the KL-randomness problem}

\author{Bj{\o}rn Kjos-Hanssen \inst{1}\orcidID{0000-0002-1825-0097}\thanks{This work was partially supported by a grant from the Simons Foundation (\#704836 to Bj\o rn Kjos-Hanssen). The authors would like to thank Reviewer 2 for \Cref{thm:reviewer-2}. The second author would also like to thank Carl Eadler for pointing him towards results about Karnaugh maps, which were helpful in thinking about \Cref{btt}.} \and
David J. Webb\inst{1}\orcidID{0000-0002-5031-7669}
}
\authorrunning{Kjos-Hanssen and Webb}

\tocauthor{Main Author, Author Two, Another Author}
\institute{
University of Hawai\textquoteleft i at M\=anoa, Honolulu HI 96822, U.S.A.,\\
\email{bjoern.kjos-hanssen@hawaii.edu, dwebb42@hawaii.edu},\\ WWW home page:
\texttt{http://math.hawaii.edu/wordpress/bjoern/}
}

\date{January 2022}
\maketitle
\begin{abstract}
While it is not known whether each real that is Kolmogorov-Loveland random is Martin-L\"of random, i.e., whether $\KLR\subseteq\MLR$,
Kjos-Hanssen and Webb (2021) showed that MLR is truth-table Medvedev reducible ($\le_{s,tt}$) to KLR.
They did this by studying a natural class Either(MLR) and showing that $\MLR\le_{s,tt}\Either(\MLR)\supseteq\KLR$.
We show that Degtev's stronger reducibilities (positive and linear) do not suffice for the reduction of MLR to Either(MLR), and some related results.
\end{abstract}

\section{Introduction}
	The theory of algorithmic randomness attempts to study randomness of not just random variables, but individual outcomes. The idea is to use computability theory and declare that an outcome is random if it ``looks random to any computer''. This can be made precise in several ways (with notions such as Martin-L\"of randomness and Schnorr randomness), whose interrelation is for the most part well understood \cite{MR2732288,MR2548883}. However, a remaining major open problem of algorithmic randomness asks whether each Kolmogorov--Loveland random (KL-random) real is Martin-L\"of random (ML-random).

	It is known that one can compute an ML-random real from a KL-random real \cite{MR2183813} and even uniformly so \cite{KHW}. This uniform computation succeeds in an environment of uncertainty, however: one of the two halves of the KL-random real is already ML-random and we can uniformly stitch together a ML-random without knowing which half. In this article we pursue this uncertainty and are concerned with uniform reducibility when information has been hidden in a sense.
	Namely, for any class of reals $\C\subseteq 2^\omega$, we write
	\[
		\Either(\C) = \{A\oplus B : A\in \C \text{ or }B\in \C\},
	\]
	where $A\oplus B$ is the computability-theoretic join: $$A\oplus B = \{2k\mid k\in A\}\cup \{2k+1\mid k\in B\}.$$
	For notation, we often refer to `even' bits of such a real as those coming from $A$, and `odd' bits coming from $B$.
	
	An element of $\Either(\C)$ has an element of $\C$ available within it, although in a hidden way. We are not aware of the $\Either$ operator being studied in the literature, although
	Higuchi and Kihara \cite[Lemma 4]{HIGUCHI20141201} (see also \cite{HIGUCHI20141058}) considered the somewhat more general operation $f(\C,\D)=(2^\omega\oplus\C)\cup (\D\oplus 2^\omega)$.

    A real $A$ is Martin-L\"of random iff there is a positive constant $c$  so that for any $n$,
	the Kolmogorov complexity of the first $n$ bits of $A$ is at least $n-c$, (that is, $\forall n,\ K(A_i\upto n)\geq n-c$).
	
	This is one of several equivalent definitions -- for instance, $A$ is Martin-L\"of random ($A\in\MLR$) iff no c.e.\ martingale succeeds on it. In contrast, $A$ is Kolomogorov-Loveland random ($A\in\KLR$) iff no computable nonmonotonic betting strategy succeeds on it.

	As $\MLR\subseteq\KLR$, KLR is trivially Medvedev reducible to MLR. %cite something
	In \cite{KHW}, $\Either$ is implicitly used to show the reverse, that MLR is Medvedev reducible to KLR. For a reducibility $r$, such as $r=tt$ (truth-table, \Cref{def:tt}) or $r=T$ (Turing), let $\le_{s,r}$ denote strong (Medvedev) reducibility using $r$-reductions, and $\le_{w,r}$ the corresponding weak (Muchnik) reducibility.
	\begin{theorem}\label{ref-cie}
	    $\MLR\le_{s,tt}\Either(\MLR)$.
	\end{theorem}
	\begin{proof}
	    \cite[Theorem 2]{KHW} shows that $\MLR\le_{s,tt}\KLR$. The proof demonstrates that $\MLR\le_{s,tt}\Either(\MLR)$ and notes, by citation to \cite{MR2183813}, that $\KLR\subseteq\Either(\MLR)$. \qed
	\end{proof}
	In fact, the proof shows that the two are truth-table Medvedev equivalent. A natural question is whether they are Medvedev equivalent under any stronger reducibility.

    Let $\DIM_{1/2}$ be the class of all reals of effective Hausdorff dimension 1/2. 
    \Cref{ref-cie} is a counterpoint to Miller's result $\MLR\not\le_{w,T}\DIM_{1,2}$ \cite{ExtractInfo}, since $\MLR\not\le_{s,tt}\DIM_{1,2}\supseteq\Either(\MLR)$.

	\begin{definition}\label{def:tt}
		Let $\{\sigma_n\mid n\in\omega\}$ be a uniformly computable list of all the finite propositional formulas in variables $v_1,v_2,\dots$.
		Let the variables in $\sigma_n$ be $v_{n_1},\dots,v_{n_d}$ where $d$ depends on $n$.
		We say that $X\models \sigma_n$ if $\sigma_n$ is true with 
		$X(n_1),\dots,X(n_d)$ substituted for $v_{n_1},\dots,v_{n_d}$.
		A reduction $\Phi^X$ is a \textbf{truth-table} reduction if there is a computable function $f$ such that for each $n$ and $X$, $n\in\Phi^X$ iff $X\models \sigma_{f(n)}$. 
	\end{definition}
	For two classes of reals $\C,\D$, we write $\C\le_{s,*}\D$ to mean that there is a $*$-reduction $\Phi$ such that $\Phi^D\in\C$ for each $D\in\D$, where $*$ is a subscript in \Cref{fig:table}. %In this notation, the aforementioned result of \cite{KHW} is that $\MLR\leq_{s,tt}\KLR$. %This is now above

	As shown in \Cref{fig:Degtev}, the next three candidates to strengthen the result (by weakening the notion of reduction under consideration) are the positive, linear, and bounded truth-table reducibilties. Unfortunately, any proof technique using $\Either$ will no longer work, as for these weaker reducibilities, $\MLR$ is not Medvedev reducible to $\Either(\MLR)$.

\section{The Failure of Weaker Reducibilities}
	When discussing the variables in a table $\sigma_{f(n)}$, we say that a variable is of a certain parity if its index is of that parity, i.e. $n_2$ is an even variable. As our reductions operate on $2^\omega$, we identify the values $X(n_i)$ with truth values as $1=\top$  and $0=\bot$.

	\begin{definition}
		A truth-table reduction $\Phi^X$ is a \textbf{positive} reduction if the only connectives in each $\sigma_f(n)$ are $\lor$ and $\land$.
	\end{definition}
	\begin{theorem}\label{positive}
		$\MLR\not\le_{s,p}\Either(\MLR)$.
	\end{theorem}
	\begin{proof}
		Let $\Phi^X$ be a positive reduction. By definition, for each input $n$, $\sigma_{f(n)}$ can be written in conjunctive normal form:
		$\sigma_{f(n)} = \bigwedge_{k=1}^{t_n} \bigvee_{i=1}^{m_k}v_{f(n),i,k}$. We say that a clause of $\sigma_{f(n)}$ is a disjunct $\bigvee_{i=1}^{m_k}v_{f(n),i,k}$. There are two cases to consider:\\
		\\
		\noindent\emph{Case 1:} There is a parity such that there are infinitely many $n$ such that every clause of $\sigma_{f(n)}$ contains a variable.

		Without loss of generality, consider the even case. Let $A = \omega\oplus R$ for $R$ an arbitrary random real.
		Each $\bigvee_{i=1}^{m_k}v_{n,i,k}$ that contains an even variable is true.
		So for the infinitely many $n$ whose disjunctions all query an even variable, $\sigma_{f(n)} = \bigwedge_{k=1}^{t_n} \top = \top$.
		As these infinitely many $n$ can be found computably, $\Phi^A$ is not immune, and so not random.\\
		%\emph{This part works even if there are negative variables.}

		\noindent\emph{Case 2:} For either parity, for almost all inputs $n$, there is a clause of $\sigma_{f(n)}$ containing only variables of that parity.

		Set $A = R\oplus \emptyset$ for an arbitrary random real $R$. For almost all inputs, some clause is a disjunction of $\bot$, so that the entire conjunction is false.
		Thus $\Phi^A$ is cofinitely often 0, and hence computable, and so not random.\qed
		%\emph{This part uses the fact that all variables are positive.}
	\end{proof}
	\begin{table}
		\centering
		\begin{tabular}{c|c|c}
			
			Reducibility	&Subscript	&Connectives\\
			\hline
			truth table		&$tt$		&any\\
			bounded $tt$	&$btt$		&any\\
			$btt(1)$		&$btt(1)$	&$\{\lnot\}$\\
			linear			&$\ell$		&$\{+\}$\\
			positive		&$p$		&$\{\land,\lor\}$\\
			conjunctive		&$c$		&$\{\land\}$\\
			disjunctive		&$d$		&$\{\lor\}$\\
			many-one		&$m$		&none\\
			
		\end{tabular}
		\caption{Correspondences between reducibilities and sets in Post's Lattice. Here $+$ is addition mod 2 (also commonly written XOR).
			Note that while a $btt$ reduction can use any connectives, there is a bound $c$ on how many variables each $\sigma_{f(n)}$ can have, hence if $c=1$ the only connective available is $\lnot$.}
		\label{fig:table}
	\end{table}
	\begin{figure}
		\centerline{
		\xymatrix{
			&&d\ar[r]&p\ar[dr]\\
			1\ar[r]&m\ar[ur]\ar[r]\ar[dr]&c\ar[ur]&\ell\ar[r]&tt\ar[r]&T\\
			&&btt(1)\ar[ur]\ar[r]&btt\ar[ur]
		}}
		\caption{\cite{OdiRed} The relationships between reducibilities in \Cref{fig:table}, which themselves are between $\leq_1$ and $\leq_T$.
			Here $x\rightarrow y$ indicates that if two reals $A$ and $B$ enjoy $A\leq_x B$, then also $A\leq_y B$. 
		}
		\label{fig:Degtev}
	\end{figure}
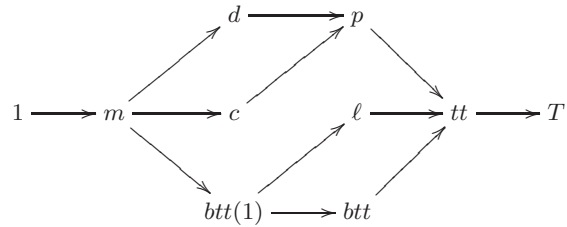
	
	\begin{remark}
		The proof of \Cref{positive} also applies to randomness over $3^{\omega}$ (and beyond).
		To see this, we consider the alphabet $\{0,1,2\}$ and let each $p(j)$ be an identity function and $\vee,\wedge$ be the maximum and minimum under the ordering $0<1<2$.
	\end{remark}

	\begin{definition}
		A truth-table reduction $\Phi^X$ is a \textbf{linear} reduction if
		each $\sigma_{f(n)}$ is of the form $\sigma_{f(n)} = \sum_{k=1}^{t_n} v_{f(n),k}$ or $\sigma_{f(n)} = 1+\sum_{k=1}^{t_n} v_{f(n),k} $ where addition is mod 2.
	\end{definition}
	\begin{theorem}\label{linear}
		$\MLR\not\le_{s,\ell}\Either(\MLR)$.
	\end{theorem}
	\begin{proof}
		We may assume that $\Phi$ infinitely often queries a bit that it has not queried before (else $\Phi^A$ is always computable).
		Without loss of generality, suppose $\Phi$ infinitely often queries an even bit it has not queried before.
		We construct $A$ in stages, beginning with $A_0 = \emptyset\oplus R$ for $R$ an arbitrary random real.

		For the infinitely many $n_i$ that query an unqueried even bit, let $v_i$ be the least such bit.
		Then at stage $s+1$, set $v_i=1$ if $\Phi^{A_s}(n_i) = 0$.
		Changing a single bit in a linear $\sigma_{f(n_i)}$ changes the output of $\sigma_{f(n_i)}$, so that $\Phi^{A}(n) = \Phi^{A_{s+1}}(n_i) = 1$.

		As these $n_i$ form a computable set, $\Phi^A$ fails to be immune, and so cannot be random.\qed
	\end{proof}

	\begin{definition}
		A truth-table reduction $\Phi^X$ is a \textbf{bounded truth-table} reduction if
		there is a $c$ such that there are most $c$ variables in each $\sigma_{f(n)}$ (in particular we say it is a \textbf{$btt(c)$} reduction).
	\end{definition}
	\begin{theorem}\label{btt}
		$\MLR\not\le_{s,btt}\Either(\MLR)$.
	\end{theorem}
	\begin{proof}
		Suppose that $\Phi$ is a $btt$-reduction from $\Either(\MLR)$ to $\MLR$ and let $c$ be its bound on the number of oracle bits queried.
		We proceed by induction on $c$, working to show that an $X=X_0\oplus X_1$ exists with $X_0$ or $X_1$ ML-random, for which $\Phi^X$ is not bi-immune.\\

		\noindent\emph{Base for the induction ($c=1$).} As $btt(1)$ reductions are linear, it is enough to appeal to \Cref{linear}.
		But as a warmup for what follows, we shall prove this case directly.
		Let $\Phi$ be a $btt(1)$ reduction. Here $\Phi^X(n)=f_n(X(q(n))$ where $f_n:\{0,1\}\to\{0,1\}$, $q:\omega\to\omega$ is computable, and $\{f_n\}_{n\in\omega}$ is computable.
		(If no bits are queried on input $n$, let $f_n$ be the appropriate constant function.)

		If for infinitely many $n$, $f_n$ is the constant function $1$ or $0$, and the claim is obvious.

		Instead, suppose $f_n$ is only constant finitely often, i.e. $f_n(x) = x$ or $f_n(x) = 1-x$ cofinitely often.
		Without loss of generality, there are infinitely many $n$ such that $q(n)$ is even. Let $X = \emptyset\oplus R$, where $R$ is an arbitrary ML-random set.

		As $X(q(n)) = 0$ and $f(x)$ is either identity or $1-x$ infinitely often,
		there is an infinite computable subset of either $\Phi^X$ or $\overline{\Phi^{X}}$ so $\Phi^X$ is not bi-immune.\\

		\noindent\emph{Induction step.} Assume the $c-1$ case, and consider a $btt(c)$ reduction $\Phi$.

		Now there are uniformly computable finite sets $Q(n)=\{q_1(n),\dots,q_{d_n}(n)\}$ and Boolean functions $f_n:\{0,1\}^{d_n}\to \{0,1\}$ such that for all $n$,
		$\Phi^X(n)=f_n(X({q_1(n)}),\dots,X(q_{d_n}(n)))$ and $d_n\le c$.

		Consider the greedy algorithm that tries to find a collection of pairwise disjoint $Q(n_i)$ as follows:
		\begin{itemize}
			\item[-] $n_0=0$.
			\item[-] $n_{i+1}$ is the least $n$ such that $Q(n)\cap\bigcup_{k<i}Q(n_k) = \emptyset$.
		\end{itemize}

		If this algorithm cannot find an infinite sequence, let $i$ be least such that $n_{i+1}$ is undefined, and define $H = \bigcup_{k\leq i}Q(n_k)$.
		It must be that for $n>n_i$ no intersection $Q(n)\cap H$ is empty.
		Thus there are finitely many bits that are in infinitely many of these intersections, and so are queried infinitely often.
		We will ``hard code" the bits of $H$ as $0$ in a new function $\hat{\Phi}$.

		To that end, define $\hat{Q}(n) = Q(n)\setminus H$, and let $\hat f$ be the function that outputs the same truth tables as $f$, but for all $n\in H$, $v_{n}$ is replaced with $\bot$.
		List the elements of $\hat{Q}$ in increasing order as $\{\hat{q}_1(n), \dots, \hat{q}_{e_n}(n)\}$.
		Now if $X\cap H = \emptyset$, any $q_i(n)\in H$ have $X(q_i(n)) = 0$, so that $\Phi^X = \hat\Phi^X$, as for every $n$,
		\[
			f(X(q_1(n)),\dots X(q_{d_n}(n)))= \hat{f_n}(X(\hat{q}_1(n)),\dots,X(\hat{q}_{e_n}(n)).
		\]

		As $Q$ and the $f_n$ are uniformly computable and $H$ is finite, $\hat{Q}$ and the $\hat{f}_n$ are also uniformly computable.
		As no intersection $Q(n)\cap H$ was empty, $e_n < d_n \leq c$. So $\hat{Q}$ and the $\hat{f}_n$ define a $btt(c-1)$-reduction.
		By the induction hypothesis, there is a real $A\in \Either(\MLR)$ such that $\hat\Phi^A$ is not random.
		$\Either(\MLR)$ is closed under finite differences (as $\MLR$ is), so the set $B = A\setminus H$ witnesses $\Phi^B = \hat\Phi^A$, and $\Phi^B$ is not random as desired.

		This leaves the case where the algorithm enumerates a sequence of pairwise disjoint $Q(n_i)$.

		Say that a collection of bits $C(n)\subseteq Q(n)$ can \emph{control} the computation $\Phi^X(n)$ if
		there is a way to assign the bits in $C_n$ so that $\Phi^X(n)$ is the same no matter what the other bits in $Q(n)$ are.
		For example, $(a \land b)\lor c$ can be controlled by $\{a, b\}$, by setting $a=b=1$.
		Note that if the bits in $C(n)$ are assigned appropriately, $\Phi^X(n)$ is the same regardless of what the rest of $X$ looks like.

		Suppose now that there are infinitely many $n_i$ such that some $C(n_i)$ containing only even bits controls $\Phi^X(n_i)$.
		Collect these $n_i$ into a set $E$. Let $X_1$ be an arbitrary ML-random set.
		As there are infinitely many $n_i$, and it is computable to determine whether an assignment of bits controls $\Phi^X(n)$, $E$ is an infinite computable set.
		For $n\in E$, we can assign the bits in $Q(n)$ to control $\Phi^X(n)$, as the $Q(n)$ are mutually disjoint. Now one of the sets
		\begin{align*}
			\{n\in E\mid \Phi^X(n) = 0\}&&\text{or}&&\{n\in E\mid \Phi^X(n) = 1\}
		\end{align*} is infinite. Both are computable, so in either case $\Phi^X$ is not bi-immune.

		Now suppose that cofinitely many of the $n_i$ cannot be controlled by their even bits. Here let $X_0$ be an arbitrary ML-random set.
		For sufficiently large $n_i$, no matter the values of the even bits in $Q(n_i)$, there is a way to assign the odd bits so that $\Phi^X(n_i) = 1$.
		By pairwise disjointness, we can assign the odd bits of $\bigcup Q(n_i)$ as needed to ensure this, and assign the rest of the odd bits of $X$ however we wish.
		Now the $n_i$ witness the failure of $\Phi^X$ to be immune. \qed
	\end{proof}

\section{Arbitrarily many columns}
	It is worth considering direct sums with more than two summands.
	In this new setting, we first prove the analog of Theorem 2 of \cite{KHW} for more than two columns,
	before sketching the modifications necessary to prove analogues of \Cref{positive,linear,btt}.

	Recall that the infinite direct sum $\bigoplus_{i=0}^\omega A_i$ is defined as $\{\langle i, n\rangle\mid n\in A_i\}$, where
	$\langle \cdot, \cdot\rangle:\omega^2\rightarrow\omega$ is a fixed computable bijection.

	\begin{definition}
		For each $\C\subseteq 2^\omega$ and ordinal $\alpha\leq\omega$, define
		\begin{eqnarray*}
		\Some(\C, \alpha) &=& \left\{\bigoplus_{i=0}^\alpha A_i \in 2^\omega\middle| \exists i\ A_i\in\mathcal{C}\right\},
		\\
		\Many(\C) &=& \left\{\bigoplus_{i=0}^\omega A_i \in 2^\omega\middle| \exists^\infty i\ A_i\in\mathcal{C}\right\}.
		\end{eqnarray*}
	\end{definition}

	These represent different ways to generalize $\Either(\C)$ to the infinite setting:
	we may know that some possibly finite number of columns $A_i$ are in $\C$, or that infinitely many columns are in $\C$. If $\alpha=\omega$, these notions are $m$-equivalent, so we can restrict our attention to $\Some(\MLR, \alpha)$ without loss of generality:

\begin{theorem}[due to Reviewer 2]\label{thm:reviewer-2}
    $\Some(\C, \omega)\equiv_{s,m}\Many(\C)$.
\end{theorem}
\begin{proof}
    The $\leq_{s, m}$ direction follows from the inclusion $\Many(\C)\subseteq\Some(\C, \omega)$.
    
    For $\geq_{s, m}$, %, consider the reduction that $i$-th column maps to $\langle i,j\rangle$-th column for all $j$.
    let $B\in\Some(\C,\omega)$ and define $A$ by:
    \[\langle \langle i,j\rangle,n\rangle \in A \iff \langle i,n\rangle\in B\]
    Then $A\le_m B$ and $A\in\Many(\C)$, so that $\Some(\C, \omega)\geq_{s, m}\Many(\C)$. \qed
\end{proof}

% For fixed $n$, $\Some(\MLR,n)\ge_{s,m}\Many(\MLR)$.
% For dissertation: $\Some(\C, \omega)\equiv_{s,1}\Many(\C)$ via infinite splitting of a column
%Not refereed: For fixed n, $\Some(\MLR,n)\not\le_{s,tt}\Many(\MLR)$

\subsection{Truth-Table Reducibility}
	Recall that a real $A$ is Martin-L\"of random iff there is a positive constant $c$ (the randomness deficiency) so that for any $n$, $K(A_i\upto n)\geq n-c$).
	Let $K_s(\sigma)$ be a computable, non-increasing approximation of $K(\sigma)$ at stages $s\in\omega$.

	\begin{theorem}\label{ttSome}
		For all ordinals $\alpha\leq\omega$, $\MLR\leq_{s,tt}  \Some(\MLR, \alpha)$.
	\end{theorem}
	\begin{proof}
		Given a set $A=\bigoplus_{i=0}^\alpha A_i$, we start by outputting bits from $A_0$,
		switching to the next $A_i$ whenever we notice that the smallest possible randomness deficiency increases.
		This constant $c$ depends on $s$ and changes at stage $s+1$ if
		\begin{equation}\label{Kolcond}
			(\exists n\le s+1)\quad K_{s+1}(A_i\upto n)<n-c_s.
		\end{equation}
		In detail, fix a map $\pi:\omega\rightarrow\alpha$ so that for all $y$, the preimage $\pi^{-1}(\{y\})$ is infinite. Let $n(0) = 0$, and if \Cref{Kolcond} occurs at stage $s$, set $n(s+1) = n(s) + 1$, otherwise $n(s+1) = n(s)$. Finally, define $A(s) = A_{\pi(n(s))}(s)$.
		
		As some $A_i$ is in $\MLR$, switching will only occur finitely often. So there is an stage $s$ such that for all larger $t$, $A(t) = A_i(t)$. Thus our output will have an infinite tail that is ML-random, and hence will itself be ML-random.

		To guarantee that this is a truth-table reduction, we must check that this procedure always halts.
		But this is immediate, as \Cref{Kolcond} is computable for all $s\in\omega$ and $A_i\in 2^\omega$.\qed
	\end{proof}
	
	\subsection{Positive Reducibility}
	We say that a variable is from a certain column if its index codes a location in that column, i.e. $n_k$ is from $A_i$ if $k = \langle i, n\rangle$ for some $n$.

	\begin{theorem}\label{positiveSome} For all $\alpha\leq\omega$, $\MLR\not\leq_{s, p} \Some(\MLR, \alpha)$.
	\end{theorem}
	\begin{proof}
		Let $\Phi^X$ be a positive reduction. Assume each $\sigma_{f(n)}$ is written in conjunctive normal form. We sketch the necessary changes to the proof of \Cref{positive}:

		\noindent\emph{Case 1:} There is an $i$ such that there are infinitely many $n$ such that every clause of $\sigma_f(n)$ contains a variable from $A_i$.

		Without loss of generality, let that column be $A_0=\omega$. The remaining $A_i$ can be arbitrary, as long as one of them is random.

		\noindent\emph{Case 2:} For all $i$, for almost all $n$, there is a clause in $\sigma_f(n)$ that contains no variables from $A_i$.

		In particular this holds for $i=0$, so let $A_0\in\MLR$ and the remaining $A_i=\emptyset$.\qed
	\end{proof}

\subsection{Linear Reducibility}

	\begin{theorem}\label{linearSome}
	For all $\alpha\leq\omega$, $\MLR\not\leq_{s, \ell}\Some(\MLR, \alpha)$.
	\end{theorem}
	\begin{proof}
		We may assume that $\Phi$ infinitely often queries a bit it has not queried before (else $\Phi^A$ is always computable).
		If there is an $i$ such that $\Phi$ infinitely often queries a bit of $A_i$ it has not queried before,
		the stage construction from \Cref{linear} can be carried out with $A_i$ standing in for $A_0$, and some other $A_j\in\MLR$.

		That case always occurs for $\alpha<\omega$, but may not when $\alpha = \omega$.
		That is, it may the the case that $\Phi$ only queries finitely many bits of each $A_i$.
		Letting each $A_i$ be random, these bits may be set to $0$ without affecting the randomness of any given column, so we could set $A_0\in\MLR$ while other $A_i=\emptyset$.
		\qed
	\end{proof}

\subsection{Bounded Truth-Table Reducibility}

	As $btt(1)$ reductions are linear, Theorem \ref{linearSome} provides the base case for induction arguments in the vein of \Cref{btt}.
	So for each theorem, we can focus our attention on the induction step:
	\begin{theorem}
		For all $\alpha\leq\omega$, $\Many(\MLR)\not\leq_{s, btt}\Some(\MLR, \alpha)$.
	\end{theorem}
	\begin{proof}
		In the induction step, the case where the greedy algorithm fails is unchanged. Instead, consider the case where the algorithm enumerates a sequence of pairwise disjoint $Q(n_i)$.
		If there is a column $A_j$ such that there are infinitely many $n_i$ such that some $C(n_i)$ containing only bits from $A_j$ controls $\Phi^X(n)$, then we proceed as in \Cref{btt}:
		start with some other $A_k\in\MLR$ while the remaining columns are empty.
		We can then set the bits in each $Q(n_i)$ to control $\Phi^X(n_i)$ to guarantee that $\Phi^X$ is not bi-immune.
		This only changes bits in $A_j$, not $A_k$, so the final $A\in\Some(\MLR,\alpha)$.

		This leaves the case where for each $A_j$, cofinitely many of the $n_i$ cannot be controlled by their bits in $A_j$.
		Here put $A_0\in\MLR$ and assign bits to the other columns as in \Cref{btt}.\qed
	\end{proof}

\section{On the Medvedev $m$-reducibility of MLR to (a.e.-)KLR} %{October 19--22}

	Let $\mu$ denote the Lebesgue fair-coin measure on $2^\omega$. It enjoys the familiar probabilistic properties such as
	\Cref{verywell}:
	\begin{lemma}\label{verywell}
		Let $\C,\D\subseteq 2^\omega$.
		If $\mu(\C)=1$ and $\mu(\D)=1$ then $\mu(\C\cap\D)=1$.
	\end{lemma}
	%\begin{proof}
	%	\begin{eqnarray*}
	%		\mu(\C\cap\D)=1-\mu\left(\overline{\C\cap\D}\right)&=&1-\mu\left(\overline{C} \cup \overline{D}\right)\\
	%		&\ge& 1-\left(\mu\left(\overline{C}\right)+\mu\left(\overline{D}\right)\right)
	%		=1-(0+0)=1.\ \qed
	%	\end{eqnarray*}
	%\end{proof}

	Given a randomness notion $\C$, we say that a set $A$ is \emph{almost everywhere $\C$-random}\footnote{This notion was previously defined for the class of computable randoms in \cite{AECompRand}.} if $\mu\{B\mid A\in\C^B\}=1$. The following is an easy corollary of van Lambalgen's theorem:

	\begin{theorem}
		$\aeMLR=\MLR$.
	\end{theorem}
	\begin{proof}  
	    The $\subseteq$ direction is immediate -- if $A$ is random relative to some oracle, it is random relative to having no oracle, so $A\in\MLR$.
		For the reverse, let $A\in\MLR$. If $B\in\MLR^A$, then $A\in\MLR^B$ by van Lambalgen's theorem. Thus $\mu\{B\mid A\in\MLR^B\}\geq\mu(\MLR^A) = 1$, so that $A\in\aeMLR$.
	\end{proof}

	The corresponding theorem for $\aeKLR$ and $\KLR$ is an open question, and in fact whether $\KLR$ satisfies a version of van Lambalgen's theorem is also open \cite{VanLam}. The situation can be summarized as follows: $$\aeMLR=\MLR\subseteq \aeKLR\subseteq\KLR.$$
	Here we investigate possible connections between $\MLR$ and $\aeKLR$.

	Write $f:A\to B$ to indicate that $f$ is a total function from $A$ to $B$.
	No confusion is likely if, in addition to the Lebesgue measure on $2^\omega$, $\mu$ also denotes the least number operator as follows:
	for an arithmetic predicate $R(k)$, $\mu k(R(k))$ is the least $k$ such that $R(k)$ is true.

	Let $f:\omega\to\omega$ with range $f[w]$, and define $f^{\mathrm{inv}}:f[\omega]\to\omega$ by $f^{\mathrm{inv}}(n)=\mu m(f(m)=n)$.
	Let $g:\omega\to\omega$ be defined by $g(0)=f(0)$, and
	\[
	g(n+1)=f(\mu k(f(k)>g(n))).
	\]
	If $f$ is unbounded then $g:\omega\to\omega$ is total. If in addition $f$ is $\Delta^0_1$, then so is $g$.
	In fact, if $f$ is unbounded and $\Delta^0_1$ then $g[\omega]$ is infinite and $\Delta^0_1$.

	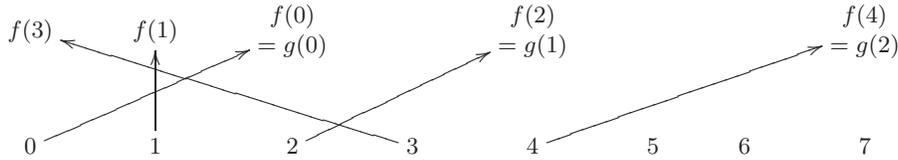
\begin{figure}
		\centering
		\[
			\xymatrix{
			f(3)&f(1) &\makecell{f(0)\\ =g(0)} & &\makecell{f(2)\\=g(1)} & & &\makecell{f(4)\\=g(2)} & &\\
			0\ar[urr] & 1\ar[u] & 2\ar[urr] & 3\ar[ulll] & 4\ar[urrr] &5&6&7
			}
		\]
		\caption{An example of the behaviors of $f$ and $g$ in \Cref{bjoernOct22}.}
		\label{fig:my_label}
	\end{figure}
	\begin{lemma}\label{bjoernOct22}
		Let $A\in 2^{\omega}$ and let $f:\omega\to\omega$ be an unbounded $\Delta^0_1$ function.
		Define $g$ by
		\[
		g(n+1)=f(\mu k(f(k)>g(n))).
		\]
		Then we have the implication $A\circ f\in\MLR\implies A\circ g\in\MLR$.
	\end{lemma}
	\begin{proof}
		Suppose that $A\circ g\not\in\MLR$. Then $A$ is Martin-L\"of null, i.e.,
		there is some uniformly $\Sigma^0_1$ class $\{U_n\}_{n\in\omega}$ with $\mu(U_n)\le 2^{-n}$ such that $A\circ g\in\bigcap_{n\in\omega}U_n$.
		Note that $f^{\mathrm{inv}}\circ g:\omega\to\omega$ is total and strictly increasing.
		Also, $f\circ f^{\mathrm{inv}}$ is the identity function on $f[\omega]\supseteq g[\omega]$.

		Thus
		\[
			A\circ g = A\circ f\circ f^{\mathrm{inv}}\circ g \]
		and we have
		\[
			A\circ f \in V_n :=\{B\mid B\circ f^{\mathrm{inv}}\circ g\in U_n\}
		\]
		The sets $V_n$ are also $\Sigma^0_1$ uniformly in $n$, and $\mu( V_n)\le 2^{-n}$. Thus $A\circ f\not\in\MLR$.
	\end{proof}

	\begin{lemma}\label{bjoernOct22two}
		If $\C$ is a class of reals and $\MLR\le_{s,m}\C$, then there is a 1-reduction
		$A\mapsto \Psi^A=A\circ g$, of $\MLR$ to $\C$, such that $g$ is strictly increasing and has computable range.
	\end{lemma}
	\begin{proof}
		Suppose $\Phi^A(n)=A\circ f(n)$ for all $A$ and $n$, where
		$A\in\C\implies\Phi^A\in\MLR$ for all $A$ and $f:\omega\to\omega$ is $\Delta^0_1$.
		Since there is no computable element of $\MLR$, it follows that $f$ is unbounded.

		Let $\Psi^A(n)=A\circ g(n)$ with $g$ as in \Cref{bjoernOct22}. Then we have
		\[
			A\in\C\implies\Phi^A\in\MLR\implies\Psi^A\in\MLR,
		\]
		as desired.
	\end{proof}

	\begin{theorem}\label{bjoernOct22three}
	If $\MLR \le_{s,m} \KLR$
	then
	$\aeKLR = \MLR$.
	\end{theorem}
	\begin{proof}
		Assume that $\MLR \le_{s,m} \KLR$. By \Cref{bjoernOct22two}, we have in fact
		a 1-reduction $\Phi$ given by $\Phi^A(n)=A(f(n))$ for an injective and strictly increasing computable $f$ with computable range $Z$.

		Let $A\in\aeKLR$.
		Consider the following classes of reals:
		\begin{eqnarray*}
		\C&=&\{B\mid A\in\KLR^B\}\\
		\D&=&\{B\mid B\in\KLR^A\}\\
		\mathcal E&=&\{B\mid B\in\MLR^A\}
		\end{eqnarray*}
		Since $A\in\aeKLR$, $\mu(\C)=1$.
		It is well-known that $\mu(\mathcal E)=1$.
		Since $\D\supseteq \mathcal E$, it follows that $\mu(\D)=1$.
		By \Cref{verywell}, $\mu(\C\cap\D)=1$; in particular, $\C\cap\D\ne\emptyset$.
		Let $B\in\C\cap\D$. Thus $A\in\KLR^B$ and $B\in\KLR^A$.
		By \cite[Proposition 11]{MR2183813}, $A\oplus_Z B\in\KLR$,
		where $A\oplus_Z B$ is the unique real whose restrictions to $\overline{Z}$ and $Z$ are $A$ and $B$, respectively.
		By definition of $Z$, we have $\Phi^{U\oplus_Z V}=U$ for all $U$ and $V$.
		Since $\Phi$ is a reduction of $\MLR$ to $\KLR$, it follows that $A=\Phi^{A\oplus_Z B}\in\MLR$, as desired.
	\end{proof}

	\begin{theorem}\label{bjoernOct22four}
		Let $Z$ be an infinite computable set and let $A,B,X$ be sets.
		If $A\in\KLR^{B\oplus X}$ and $B\in\KLR^{A\oplus X}$ then $A\oplus_Z B\in\KLR^X$.
	\end{theorem}
	\begin{proof}
		Relativization of \cite[Proposition 11]{MR2183813}.
	\end{proof}

	We can relativize to obtain the following conditional result, a strengthening of \Cref{bjoernOct22three}:
	\begin{theorem}
		If $\MLR\le_{s,m}\aeKLR$ then $\MLR=\aeKLR$.
	\end{theorem}
	\begin{proof}
		Assume $\MLR\le_{s,m}\aeKLR$ as witnessed by a reduction $\Phi$, and let $A\in\aeKLR$. Let $Z$ be as in the proof of \Cref{bjoernOct22three}.

		By definition of $\aeKLR$, for almost all $(B,X)$, $A\in\KLR^{B\oplus X}$. Moreover, for almost all $(B,X)$, $B\in\MLR^{A\oplus X}\subseteq\KLR^{A\oplus X}$.

		Therefore
		\[
			\mu\{(B,X)\mid A\oplus_Z B\in\KLR^X\}=1
		\]
		by \Cref{bjoernOct22four}. By Fubini's Theorem this implies
		\[
			\mu\{ B\mid \mu\{X\mid A\oplus B\in\KLR^X\}=1\}=1.
		\]
		Since measure-one sets are nonempty, there exists a set $B$ such that
		\[
			\mu\{X\mid A\oplus B\in\KLR^X\}=1,
		\]
		i.e., $A\oplus B\in\aeKLR$.
		By assumption on $\Phi$, $A=\Phi^{A\oplus_Z B}\in\MLR$, as desired.
	\end{proof}

	We can also strengthen these results to work for $btt(1)$-reducibility instead of $m$-reducibility.

	If an analogue of van Lambalgen's theorem holds for KLR (as it does for MLR), then the two theorems have the same content, as KLR = a.e.-KLR.

\bibliographystyle{alpha}
\bibliography{cie2022}
\end{document}